\documentclass[11pt]{article}
\usepackage{amscd,amssymb,amsmath}
\usepackage[utf8]{inputenc}
\usepackage{rotating}
\usepackage[top=2.6cm, bottom=3.0cm, left=2.4cm, right=2.4cm]{geometry}
\usepackage{dsfont}
\usepackage{graphicx} 
\usepackage{enumerate}
\usepackage{color}

\usepackage{amsfonts,amsmath,graphicx,amssymb}

\def\kernel{\zeta}
\def\odkernel{\eta}
\def\odkernelbase{\phi}
\def\solution{\psi}

\def\R{\mathbb{R}}

\def\N{\mathbb{N}}
\def\Z{\mathbb{Z}}

\def\cK{\mathcal{K}}

\def\00{\mathbf{0}}
\def\11{\mathbf{1}}
\def\22{\mathbf{2}}
\def\33{\mathbf{3}}
\def\44{\mathbf{4}}
\def\55{\mathbf{5}}
\def\66{\mathbf{6}}
\def\77{\mathbf{7}}
\def\88{\mathbf{8}}
\def\99{\mathbf{9}}

\newtheorem{theorem}{Theorem}[section]

\newtheorem{lemma}[theorem]{Lemma}
\newtheorem{proposition}[theorem]{Proposition}
\newtheorem{definition}[theorem]{Definition}

\newtheorem{remark}[theorem]{Remark}

\def\qed{\hfill$\Box$} 
\newenvironment{proof}{{\noindent \bf 
Proof:}}{\hfill\qed\bigskip}

\begin{document}

\renewcommand{\theequation}{\arabic{section}.\arabic{equation}}

\title{A kernel-based discretisation method for first order
  partial differential equations of evolution type} 

\author{Tobias
  Ramming\thanks{tobias.ramming@uni-bayreuth.de}\\Department of
  Mathematics\\ University 
    of Bayreuth\\ D-95440 Bayreuth\\ Germany
\and Holger Wendland\thanks{holger.wendland@uni-bayreuth.de}\\
     Department of Mathematics\\
        University of Bayreuth\\
        D-95440 Bayreuth\\ Germany
}

\maketitle
\begin{abstract}
We derive a new discretisation method for first
order PDEs of arbitrary spatial dimension, which is based upon a
meshfree spatial approximation. This spatial approximation is similar
to the SPH (smoothed particle 
hydrodynamics) technique and is a typical kernel-based method. It
differs, however, significantly from the SPH method since it employs
an Eulerian and not a Lagrangian approach. We
prove stability and convergence for the resulting semi-discrete scheme under
certain smoothness assumptions on the defining function of the PDE. The
approximation order depends on the underlying kernel and the
smoothness of the solution. Hence, we also review an easy way of
constructing smooth kernels yielding arbitrary convergence
orders. Finally, we give a numerical example by testing our method in
the case of a one-dimensional Burgers equation. 
\end{abstract}

\section{Introduction}

In this paper, we will derive and analyse a new discretisation method for
a large class of first order evolution equations, i.e. we are
interested in finding approximate solutions to initial value problems of
the form 
 \begin{eqnarray}
\partial_t \rho + f(t,x,\rho,\nabla\rho) &=&0\qquad \text{on }(0,\infty)
  \times\R^n \label{problem_equation}\\  
\rho(0,\cdot)&=&\rho_0\qquad \text{on }\R^n. \label{problem_initialValue}
\end{eqnarray}
Here, $f:\R^{2n+2}\to\R$ is a given, twice continuously differentiable
mapping, $\rho_0\in C^r(\R^n)$ is the given initial condition and
$\nabla\rho$ denotes as usual the vector of first order spatial
derivatives of $\rho$. The function $\rho:[0,\infty)\times\R^n\to\R$
is the solution we want to compute and we will assume that the above
problem has a solution $\rho\in C^{1,r}([0,\infty)\times\R^n)$, i.e.~a solution which 
has at least first order continuous derivatives in time and $r$-th order continuous derivatives in space with $r\ge 1$. 

These somewhat strong conditions are required for our error 
analysis. The numerical scheme itself can be set up under 
much milder conditions. Nonetheless, 
as usual the efficiency of 
the scheme is based upon
the assumption that we have a strong solution to the problem.

First order problems of the above type occur in many
different situations, often when modelling physical phenomena. The
most simple example is given by the well known {\em linear transport
  problems} 
\[
\partial_t \rho + u\cdot \nabla \rho = g
\]
with a given drift $u:[0,\infty)\times\R^n\to\R^n$ and a source term
  $g:[0,\infty)\times\R^n\to\R$. Further examples range from the
{\em nonlinear evolution problems
      of Hamilton-Jacobi type}, given by equations of the form
\[
\partial_t \rho + H(x,\nabla \rho) = 0
\]
with given Hamiltonian function $H:\R^{2n}\to\R$ to problems from
{\em optimal control and dynamical programming}, where the control usually 
also has to satisfy a given (often physically motivated) partial 
differential equation.

One specific possible application we have in mind is the determination
of the basin of attraction of a system of ordinary differential
equations $\dot{x}(t)=f(t,x(t))$. Here, a feasible approach is to set
up a first order partial differential equation for the Lyapunov
function $L(x,t)$, which then has to satisfy
\[
\nabla L(t,x)\cdot f(t,x) + \partial_t L(t,x) = -g(t,x)
\]
with a given function $g$. For a recent review see
\cite{Giesl-Hafstein-15-2}. Since the function $g$ can be chosen by
the user, it is possible to have a smooth solution such that this
particular type of problem is indeed covered by our convergence
analysis below.

Numerically, equations of the form (\ref{problem_equation}),
(\ref{problem_initialValue})  are typically solved by classical
finite differences, semi-Lagrangian approximation schemes, level
set methods or by finite elements using an approach based upon
viscosity solutions \cite{Falcone-Ferretti-98-1, Falcone-Ferretti-14-1}.

However, some of these problems do not exhibit classical smooth
solutions. This is particularly the case when scalar, non-linear
conservation laws of the form
\[
 \partial_t \rho + \nabla\cdot f(\rho) = 0
\]
with a flux function $f:\R\rightarrow\R^n$ are considered. Here it is
well known that, even for smooth initial data, the solutions can
develop shocks in finite time (cf.~\cite{Bressan-09-1}) or an even worse
non-smooth behaviour. 

The numerical simulation of such conservation laws and the handling of
these non-smooth solutions has attracted considerable attention within
the last years. In particular weighted essentially non-oscillatory
(WENO) schemes and discontinuous Galerkin schemes have proven to work
quite well - at least  in one spatial dimension
(cf.~\cite{LeVeque-92-1, Shu-09-1, Shu-13-1} and the
references therein). However, there are still considerable problems   
when it comes to treating non-smooth data in the multivariate setting,
since most of the known schemes are more or less extensions of the
one-dimensional  schemes and show spurious effects if the shocks and
other discontinuities are not aligned with the mesh
\cite{Eymann-Roe-13-1}. As a consequence, there is a need for
suitable, genuinely multivariate methods. While the method we want to
analyse in this paper is still far from tackling these kind of
problems, the method can hopefully be extended in such a direction in
the future.

In this paper, we propose a new spatial discretisation method, which,
in a certain way, uses techniques from classical particle methods such
as SPH (smoothed 
particle hydrodynamics), see for example
\cite{Benmoussa-Vila-00-1,Violeau-12-1,Vila-99-1,Benmoussa-06-1,
  Monaghan-05-1,Monaghan-12-1}. For this reason, our analysis  will partially
employ ideas of earlier works such as \cite{Cortez-97-1,
  Raviart-85-1}. However, our method differs from these methods
significantly since we do not use a Lagrangian approach as it is
usually done in this context. Instead, we employ an Eulerian 
approach. Hence, in a certain way, our method can also be seen as a
generalised finite difference method. 

While our analysis is based upon the fact that our spatial
discretisation points form a regular grid, it is our goal to extend
these results to arbitrary point sets later on.

The paper is organised as follows.  In the next section,
we will derive our discretisation scheme and state our main
convergence result. The third section is devoted to some approximation
results which extend results from \cite{Raviart-85-1}. The fourth
section then deals with the proof of our main convergence result. In
the fifth section, we give a general scheme to construct high order
kernels. The final section is devoted to a numerical example.

In this paper, we will only consider problems on all of $\R^n$ as the
spatial domain. From a practical point of view, this means that we
will, where necessary, assume that for a
fixed time interval $[0,T]$ there is a compact set
$\Omega\subseteq\R^n$ such that  the solution $\rho(t)=\rho(t,\cdot)$ of
(\ref{problem_equation}), (\ref{problem_initialValue}) has support
contained in $\Omega$ for all $t\in[0,T]$. We will choose $\Omega$
large enough such that no boundary effects occur. For the purpose of
our analysis, we will without restriction also assume that $\Omega$ is
convex. In this situation, we obviously only need to know the
defining function $f$ on the set
\begin{equation}\label{M}
M:=\{(t,x,\rho(t,x),\nabla \rho(t,x))\in\R^{2n+2} : t\in[0,T],
x\in\Omega\}\subseteq\R^{2n+2}, 
\end{equation}
which is also compact since $\rho$ is supposed to be a $C^1$
function. Again, for the purpose of our analysis, we will extend this
set $M$ to a bigger, convex set 
\begin{equation}\label{Mtilde}
\widetilde{M}:=[0,T]\times\Omega_1\times \Pi.
\end{equation}
Here $\Omega_1$ is the convex hull of $\cup_{x\in\Omega} B_1(x)$,
where $B_1(x)$ is the ball of radius $1$ and centre $x$. Moreover,
$\Pi$ is a convex, compact super set of $\{(\rho(t,x),\nabla
\rho(t,x)) : t\in[0,T], x\in\Omega\}\subseteq \R^{n+1}$.

Then we may assume that $f$ is smoothly defined on all of $\R^n$ but
has support in $\widetilde{M}$. Obviously, this is no restriction under
the given assumptions and we can modify any given $f$ that does not
satisfy this condition accordingly.


As usual, $W_p^k(\R^n)$ will denote the space of all functions having
weak derivatives up to order $k$ in $L_p$. Its norm will be denoted by
$\|\cdot\|_{W_p^k(\R^n)}$ and the semi-norm consisting only of the
derivatives of order $|\alpha|=k$ will be denoted by $|\cdot|_{W_p^k(\R^n)}$.

\section{The Discretisation Scheme and its Convergence}

Our discretisation scheme is a classical kernel-based approximation
scheme. We will employ kernels of the following form.

\begin{definition} \label{definition_kernel}
A continuous and bounded function $\zeta:\R^n\to\R$ is called a kernel
of order $k\ge 1$ if 
\begin{enumerate}
\item [(i)] $\zeta\in L_1(\R^n)$ with $\int_{\R^n} \zeta(x)dx = 1$,
\item [(ii)] $\int_{\R^n} x^\alpha \zeta(x)dx = 0$ for all $\alpha\in\N^n_0$
  with $1\leq |\alpha| \leq k-1$,
\item [(iii)] $\int_{\R^n}|x|^k |\zeta(x)|dx < \infty$.
\end{enumerate}

 For $k,s\in\N_0$ with $k\geq 1$, we define $\cK^{k,s}$
 to be the set of all such kernels $\zeta\in C^s(\R^n)\cap
 W_1^s(\R^n)$  of order $k$ and smoothness $s$. If $s=0$ we
 require $\zeta\in C(\R^n)\cap L_1(\R^n)$. Moreover, we define
 $\cK_c^{k,s}$ to be the subset of such kernels having compact support.
\end{definition}

Note that the second condition in Definition \ref{definition_kernel}
will only become active if $k\ge 2$.

So far, $\zeta$ is rather a function than a kernel. It is named a
kernel since it is used as a convolution kernel.
To be more precise, let $\varepsilon>0$ and a kernel
$\zeta$ as before be given and let us define the scaled version
 $\zeta_\varepsilon(x)=\varepsilon^{-n}\zeta(x/\varepsilon)$,
$x\in\R^n$. Then, any function $\rho:[0,\infty)\times \R^n\to\R$ can be
approximated by a convolution of the form
\begin{equation}\label{convolution}
\rho^{\varepsilon}(t,x) = (\rho*\zeta_\varepsilon)(t,x) = \int_{\R^n}
\rho(t,y)\zeta_\varepsilon(x-y)dy, \qquad (t,x)\in[0,\infty)\times \R^n,
\end{equation}
provided that the integral exists. Note that the convolution is only
taken with respect to the spatial variable and not with respect to the
time variable.

Next, in particle methods, the approximation $\rho^\varepsilon$ is
further approximated using a quadrature rule.  From now on, we will
use the notation 
\begin{itemize}
 \item $x_i = ih$, $i\in\Z^n$,
 \item $\rho_{x_i} = \rho_{ih} = \rho(\cdot,x_i)$, $i\in\Z^n$,
\end{itemize}
where $h>0$ is a given discretisation parameter.
Thus, applying a simple composite rectangular rule to the integral in
(\ref{convolution}) yields the new approximation
\[
[\rho]_x(t) := [\rho](t,x) := h^n\sum_{j\in\Z^n}
\rho_{jh}(t)\zeta_\varepsilon(x-x_j),\qquad (t,x)\in [0,\infty) \times\R^n. 
\]
For simplicity, we  will assume here that $\rho$ and $\zeta$ are
chosen such that the series is well defined. This is, for example, the
case if $\rho$ satisfies sufficient decay conditions or if $\zeta$
is compactly supported. In this case, $[\rho]$ defines a smooth
function in space where the smoothness is determined by the smoothness
of $\zeta$. It also defines a smooth function in time, where the
smoothness is now determined by the smoothness of $\rho$ in time.

Note that $[\rho]$ is  already defined if only a countable
number of time-dependent (or even constant) functions
$\rho_j:[0,\infty)\to\R$, $j\in\Z^n$, is given and $\rho_{jh}$ is
  replaced by $\rho_j$. This obvious observation can be used to
  compute an approximate solution to (\ref{problem_equation}) by
  reducing the original problem to the problem of 
finding approximate coefficients $\rho_j^{\varepsilon
  h}:[0,\infty)\to\R$.

Hence, we may define  a function  depending on time and space by
\begin{equation}\label{sphApprox_function}  
[\rho^{\varepsilon h}]_x(t) = [\rho^{\varepsilon  h}](t,x) =
h^n\sum_{j\in\Z^n} \rho_j^{\varepsilon h}(t)\zeta_\varepsilon(x-x_j),\quad
(t,x)\in[0,\infty)\times\R^n. 
\end{equation}
Provided that the kernel $\zeta$ is at least $|\alpha|$-times continuously
differentiable, spatial derivatives of this function are simply given
by 
\[
\partial^\alpha[\rho^{\varepsilon  h}](t,x) = h^n\sum_{j\in\Z^n} \rho_j^{\varepsilon
  h}(t)\partial^\alpha \zeta_\varepsilon(x-x_j), \qquad
(t,x)\in[0,\infty)\times\R^n. 
\]
In particular, we have
\begin{equation} \label{sphApprox_gradient} 
\nabla[\rho^{\varepsilon h}]_{x} (t)=\nabla[\rho^{\varepsilon
    h}](t,x) = h^n\sum_{j\in\Z^n} \rho_j^{\varepsilon 
  h}(t)\nabla \zeta_\varepsilon(x-x_j), \qquad (t,x)\in[0,\infty)\times\R^n.
\end{equation}
In the next section, we will discuss the approximation properties of
this approximation process. The results derived there are essential
for the proof of our main result; they are mainly technical
improvements of earlier results which can be found in  
\cite{Raviart-85-1}. 

Before that, we will state our approximation method and our main
convergence result.

If we restrict (\ref{problem_equation}) to the spatial points
$x_i=ih$, $i\in\Z^n$, the initial partial differential equation reduces to the system
of differential equations
\[
\dot{\rho}_{ih}(t) = -f(t,ih,\rho(t,ih),\nabla\rho(t,ih)), \qquad
(t,i)\in(0,\infty)\times \Z^n.
\]

We can now solve the latter equation by approximating  the expressions
$\rho$ and $\nabla\rho$ on the right-hand side  by their
approximations $[\rho^{\varepsilon h}]$ and $\nabla[\rho^{\varepsilon
    h}]$, respectively.

\begin{theorem} \label{main_result} Let $T>0$ and $I:=[0,T]$.
 Assume that the solution of problem
 \eqref{problem_equation}, \eqref{problem_initialValue} for $f\in
 C^2(I\times\R^{2n+1})$ and $\rho_0\in C_c^r(\R^n)$ satisfies $\rho\in
 C_c^{2,r}(I\times\R^n)$. 
Let $\zeta\in\mathcal K^{k,s}_c$ be a kernel of order $k\ge 1$.

 For $\varepsilon, h>0$ let $\{\rho^{\varepsilon h}_i\}_{i\in\Z^n}$ be
 the solution of the initial value problem 
 \begin{eqnarray*}
\dot\rho_i^{\varepsilon h}&=&-f(., ih, [\rho^{\varepsilon h}]_{ih},
         \nabla [\rho^{\varepsilon h}]_{ih}),\\
  \rho_i^{\varepsilon h}(0)&=& \rho_0(ih)
 \end{eqnarray*}
with $[\rho^{\varepsilon h}]_{ih}$ and
$\nabla[\rho^{\varepsilon h}]_{ih}$ as in \eqref{sphApprox_function}
and \eqref{sphApprox_gradient}, respectively. If the parameters obey
the relations
\[
r\geq \max\{k+1,\ell\}, \quad k\ge 2+\frac{n}{2}, \quad s>\ell>n\quad 
\text{and} \quad h\leq \varepsilon^{1+(3+\frac{n}{2})/{\ell}}
\]
then there is a constant $C=C(f,\rho,T,\zeta)>0$ independent of
$\varepsilon$ and $h$ such that for all $\alpha\in\N_0^n$ with
$|\alpha|\leq 1$ it holds that 
\[\|\partial^\alpha\rho-\partial^\alpha[\rho^{\varepsilon h}]\|_{L_\infty(\R^n)} \leq
 C\varepsilon^{-|\alpha|-\frac{n}{2}}\left(\varepsilon^k
 +\frac{h^{\ell}}{\varepsilon^{\ell+1}}\right)\leq  
 C\varepsilon^{k-|\alpha|-\frac{n}{2}}\] 
 uniformly on $[0,T]$.

\end{theorem}

We will postpone the proof of this theorem. However, since its proof
and the proof of certain auxiliary results will use a discrete $L_p$
norm, which is quite standard in this context, we will introduce this
norm now.

\begin{definition}\label{defDiscreteNorm}
 Let $1\le p\le \infty$ and $h>0$. For a given sequence
  $(\rho_i)_{i\in\Z^n}\in\ell_p$ we define the $h$-dependent $\ell_p$-norm by
\begin{equation}\label{lph}
\|\rho\|_{p,h}:= \begin{cases}
\left(h^n\sum_{i\in\Z^n} |\rho_i|^p\right)^{1/p} & \mbox{ for } 1\le
p<\infty\\
\sup_{i\in\Z^n} |\rho_i| & \mbox{ for } p=\infty.
\end{cases}
\end{equation}
\end{definition}

Obviously,  we are particularly interested in the
situation $\rho_i=\rho(ih)$ in which $\|\cdot\|_{p,h}$ becomes an
approximation to the continuous ${L_p(\R^n)}$ norm.

\section{Auxiliary Results on Quasi-Interpolation}
\label{section_auxiliaryResults}
\setcounter{equation}{0}

The main result of this section is essential for proving our
convergence theorem, Theorem \ref{main_result}. It specifies the
approximation power of our discretisation technique for approximating
functions. It generalises an earlier result of \cite{Raviart-85-1},
particularly by providing also estimates for derivatives.

To proof this theorem, we require two auxiliary results, Theorem 3.1
and Lemma 4.4 from \cite{Raviart-85-1}.

The first result analyses the quadrature error of the composite
rectangular rule and is quite standard. Its proof is based on the
Bramble-Hilbert lemma.

\begin{lemma} \label{lemma_quadrature}
 Let $\ell\in\N$ with $\ell>n\geq 1$. Then there exists a constant $C>0$
 independent of $h$ such that for all functions $\rho\in
 W_1^{\ell}(\R^n)$ we have 
\[
  \left|\int_{\R^n}\rho(y)dy - h^n\sum_{j\in\Z^n}\rho(jh)\right| \leq C h^\ell
  |\rho|_{W_1^\ell(\R^n)}.
\]
\end{lemma}

The second result analyses the convolution error. Its proof is mainly
based upon the Taylor expansion.

\begin{lemma} \label{lemma_convolution}
 Let $\zeta\in\cK^{k,s}$ be a kernel function of order $k\geq 1$
 and $\rho\in  C^{r}(\R^n)\cap W_p^r(\R^n)$ for suitable constants $1\leq p\leq
 \infty$ and $k\leq r\in\N$. Then, there is a constant $C=C(\zeta)>0$
 such that for any $\varepsilon>0$ and $\alpha\in\N_0^n$ with
 $|\alpha|\leq r-k$ the following holds: 
\[
 \|\partial^\alpha \rho-\partial^\alpha \rho^\varepsilon\|_{L_p(\R^n)} =
 \|\partial^\alpha \rho- \rho*\partial^\alpha\zeta_\varepsilon\|_{L_p(\R^n)}
 \leq C\varepsilon^{k} |\rho|_{W_p^{k+|\alpha|}(\R^n)}. 
\]
\end{lemma}

The kernel itself satisfies the following norm estimates, which we
will also require later on.

\begin{lemma} \label{lemma_core_norm_estimate}
 Let $1\leq p,q\leq\infty$ with $\frac{1}{p}+\frac{1}{q}=1$, $s\in\N$
 and $\zeta\in W_p^{s}(\R^n)$. 
 For $\varepsilon>0$ let
 $\zeta_\varepsilon(x)=\varepsilon^{-n}\zeta(x/\varepsilon)$,
 $x\in\R^n$. Let  $\alpha\in\N_0^n$.
 \begin{enumerate}
  \item[(a)] If $|\alpha|\leq s$ then 
  \[
\|\partial^\alpha\zeta_\varepsilon\|_{L_p(\R^n)} =
\varepsilon^{-\frac{n}{q}-|\alpha|}\|\partial^\alpha\zeta\|_{L_p(\R^n)}.
\]
  \item[(b)] If $\zeta\in W_2^s(\R^n)\cap
    W_\infty^s(\R^n)$ and if $|\alpha|< s-n$ then 
    for every $p\in [2,\infty]$ there is a 
    constant $C=C(\zeta, \alpha, p)>0$ such that for every
    $\varepsilon>h>0$ we have
\begin{equation}\label{basicestimate}
\|\partial^\alpha\zeta_\varepsilon\|_{p,h} \leq
  C \varepsilon^{-\frac{n}{q}-|\alpha|}.
\end{equation}
\item [(c)] If $\zeta\in \cK_c^{k,s}$ for any $k\in\N$ then
  (\ref{basicestimate}) even holds for all $p\in[1,\infty]$.
 \end{enumerate}
\end{lemma}

\begin{proof} The first statement is obvious for $p=\infty$ and
  follows from the transformation formula for all other $p$.

The second statement also immediately follows for $p=\infty$. For $p=2$,
we note that 
Lemma \ref{lemma_quadrature} yields
\begin{eqnarray}
 \|\partial^\alpha\zeta_\varepsilon\|_{2,h}^2
 &\leq& \left|h^n\sum_{i\in\Z^n} |\partial^\alpha\zeta_\varepsilon(ih)|^2 - \int_{\R^n}
 |\partial^\alpha\zeta_\varepsilon(x)|^2 dx\right| +
 \|\partial^\alpha\zeta_\varepsilon\|_{L_2(\R^n)}^2 \nonumber\\ 
 &\leq& Ch^\ell|g|_{W_1^\ell(\R^n)} +
 \|\partial^\alpha\zeta_\varepsilon\|_{L_2(\R^n)}^2,
\label{bound1}
\end{eqnarray}
as long as  $g=|\partial^\alpha\zeta_\varepsilon|^2\in W_1^{\ell}(\R^n)$ with
 $n<\ell\in\N$ which is guaranteed by our assumptions because of $\ell:=s-|\alpha|>n$ and
\begin{eqnarray*}
 |g|_{W_1^\ell(\R^n)}
 &=& \sum_{|\beta|=\ell} \int_{\R^n}
 |\partial^{\beta}(\partial^\alpha\zeta_\varepsilon(x))^2|dx \\ 
 &\leq& \sum_{|\beta|+|\gamma|=\ell} C\int_{\R^n}
 |\partial^{\alpha+\beta}\zeta_\varepsilon(x)||
\partial^{\alpha+\gamma}\zeta_\varepsilon(x)|dx \\ 
 &\leq&\sum_{|\beta|+|\gamma|=\ell}C\|\partial^{\alpha+\beta}
\zeta_\varepsilon\|_{L_2(\R^n)}\|\partial^{\alpha+\gamma}\zeta_\varepsilon\|_{L_2(\R^n)}\\ 
 &\leq& C\varepsilon^{-n-2|\alpha|-\ell}.
\end{eqnarray*}
This gives, together with $0<h<\varepsilon$ and (\ref{bound1}),
\[
\|\partial^\alpha\zeta_\varepsilon\|_{2,h}^2 \le C( h^\ell
\varepsilon^{-n-2|\alpha|-\ell} + \varepsilon^{n-2|\alpha|}) \le C
\varepsilon^{-n-2|\alpha|}.
\]
Since we also know that $\|\partial^\alpha\zeta_\varepsilon\|_{\infty,h}
\le C \varepsilon^{-n-|\alpha|}$, we can conclude via interpolation that 
\begin{eqnarray*}
\|\partial^\alpha \zeta_\varepsilon\|_{p,h} &\le&
\|\partial^\alpha\zeta_\varepsilon\|_{2,h}^{\frac{2}{p}}\;
\|\partial^\alpha\zeta_\varepsilon\|_{\infty,h}^{1-\frac{2}{p}}\\
& \le& C \varepsilon^{-\left(\frac{n}{2}+|\alpha|\right)\frac{2}{p}} 
\varepsilon^{-(n+|\alpha|)\left( 1-\frac{2}{p}\right)}\\
& = & C \varepsilon^{-\frac{n}{q}-|\alpha|}.
\end{eqnarray*} 
Finally, if $\zeta\in\cK_c^{k,s}$, then we
automatically have $\zeta\in W_1^s(\R^n)\cap
W_\infty^s(\R^n)$. Furthermore, since the number of $i\in\Z^n$ with
$i\frac{h}{\varepsilon}$ in the support of $\zeta$ is bounded by a
constant times $(\varepsilon/h)^n$, we see that 
\[
\|\partial^\alpha\zeta_\varepsilon\|_{1,h} = h^n
\sum_{i\in\Z^n} |\partial^\alpha \zeta(ih/\varepsilon)| \le C\varepsilon^{-|\alpha|}.
\] 
For general $p$ the result then follows by interpolation between
$p=1$ and $p=\infty$ again.
\end{proof}

After these preparations, we can state and prove a result concerning the
approximation power of such quasi-interpolants. Note, that
(\ref{sphApprox_function}) can also be defined for a function
$\rho$, which does not depend on time. In this case $\rho_j$ is just a
constant for each $j\in\Z^n$. For simplicity, we state the next result
under this assumption. However, if $\rho$ depends on $t\in[0,T]$, then the
result obviously holds point-wise for all $t\in [0,T]$.

\begin{theorem} \label{theorem_convolution_quadrature}
  Let $1\leq p\leq \infty$.  Let $\rho\in C^r(\R^n)\cap W_p^r(\R^n)$
  and $\zeta\in \cK^{k,s}$ with $r,k,s\in\N$  be given. Assume that 
$r\ge k$ and $r,s\ge \ell> n$  with $\ell\in\N$. Finally, let
  $\varepsilon>h>0$. 
  Then, there is a constant $C=C(\zeta)>0$ independent of $\varepsilon$
  and $h$ such that for every $\alpha\in\N_0^n$ with $|\alpha|\leq
  \min\{r-k, s-\ell\}$ the error between
  $\rho$ and its approximation   $[\rho]=h^n\sum_i
  \rho(ih)\zeta_\varepsilon(\cdot-ih)$ can be bounded by
  \begin{equation} \label{estimate_convolution_quadrature}
    \|\partial^\alpha\rho -\partial^\alpha[\rho]\|_{L_p(\R^n)}
    \leq C \left( \varepsilon^k \|\rho\|_{W_p^{k+|\alpha|}(\R^n)} +
    \frac{h^\ell}{\varepsilon^{\ell+|\alpha|}}\|\rho\|_{W_p^\ell(\R^n)} \right) .
  \end{equation}
\end{theorem}

\begin{proof}
We can split the error on the left-hand side of
\eqref{estimate_convolution_quadrature} into a convolution and a
quadrature error:
\[
 \|\partial^\alpha\rho -\partial^\alpha[\rho]\|_{L_p(\R^n)} \leq
 \|\partial^\alpha\rho-\partial^\alpha\rho^\varepsilon\|_{L_p(\R^n)} + 
 \|\partial^\alpha\rho^\varepsilon-\partial^\alpha[\rho]\|_{L_p(\R^n)}.
\]
Using Lemma \ref{lemma_convolution}, we see that the first term on the
right-hand side can be bounded by 
\[
\|\partial^\alpha\rho-\partial^\alpha\rho^\varepsilon\|_{L_p(\R^n)}\le 
C\varepsilon^k|\rho|_{W_p^{k+|\alpha|}(\R^n)},
\]
 provided that  $r\geq
k+|\alpha|$. For the second term, we first note that we have for
$x\in\R^n$ fixed that
\[
 |\partial^\alpha\rho^\varepsilon(x)-\partial^\alpha[\rho](x)|
 =\left|\int_{\R^n}\rho(y)\partial^\alpha\zeta_\varepsilon(x-y)dy - \sum_{j\in\Z^n}
 h^n\rho(jh)\partial^\alpha\zeta_\varepsilon(x-jh) \right|. 
\]
Hence, if  $r\geq \ell$ and $s-|\alpha|\geq \ell$ for $\ell>n$ then we can use Lemma
\ref{lemma_quadrature} to derive
\begin{eqnarray*}
 |\partial^\alpha\rho^\varepsilon  (x) -\partial^\alpha[\rho](x)| 
 &\leq &Ch^\ell\left|
 \rho \,\partial^\alpha\zeta_\varepsilon(x-\cdot)\right|_{W_1^\ell(\R^n)} \\ 
 &\leq &Ch^\ell
\sum_{|\beta|=\ell}\int_{\R^n}|\partial_y^\beta(\rho(y)
\partial^\alpha\zeta_\varepsilon(x-y))|dy \\ 
 &\leq& Ch^\ell\sum_{|\beta|,|\gamma|\leq
  \ell}\int_{\R^n}|\partial^\beta\rho(y)|\,|
\partial^{\alpha+\gamma}\zeta_\varepsilon(x-y)|dy \\ 
 &=& Ch^\ell\sum_{|\beta|,|\gamma|\leq \ell}|\partial^\beta\rho|*| 
\partial^{\alpha+\gamma}\zeta_\varepsilon|(x).
\end{eqnarray*}
Young's inequality finally yields
\[
 \|\partial^\alpha\rho^\varepsilon -\partial^\alpha[\rho]\|_{L_p(\R^n)}
 \leq Ch^\ell\|\rho\|_{W_p^\ell(\R^n)}\|\zeta_\varepsilon\|_{W_1^{\ell+|\alpha|}(\R^n)} 
 \leq C(\zeta)\frac{h^\ell}{\varepsilon^{\ell+|\alpha|}}\|\rho\|_{W_p^\ell(\R^n)},
\]
where we have also used Lemma \ref{lemma_core_norm_estimate} in the
last step.
\end{proof}

\begin{remark} We will use this result in particular to bound
  estimates on  first order derivatives. Hence, the smoothness $r$ and
  $s$ of $\rho$ and $\zeta$, respectively, have to satisfy $r\ge
  \max\{\ell,k+1\}$ and $s\ge \ell+1$ with $\ell>n$.
\end{remark}

Finally, we want to state and prove a result, which is interesting on
its own. It shows that the kernel and its derivatives provide a Bessel
sequence. We use it in our proof to establish bounds on the $L_\infty$
errors of our approximation based on error bounds for the coefficients
in the discrete $L_2$ norm. It seems worthwhile to mention that for
this purpose property $ii)$ of the kernel definition, i.e. the
property that defines the order of the kernel, is not required. 

\begin{theorem} \label{theorem_convolution_quadrature_stability}
 Let $1\leq p,q\leq\infty$ with $\frac{1}{p}+\frac{1}{q}=1$. Let
 $(a_i)_{i\in\Z^n}\in \ell_q(\R)$ and let $\zeta$ be a kernel as in
 Lemma \ref{lemma_core_norm_estimate}, i.e.  either $\zeta\in
 W_2^s(\R^n)\cap  W_\infty^s(\R^n)$  for $p\in[2,\infty]$ or $\zeta\in
 \cK_c^{k,s}$ for $p\in[1,\infty]$.
 If $\alpha\in\N_0^n$ satisfies $|\alpha|<s-n$  then there is a
 constant $C>0$ such that for 
 $\varepsilon>h$ we have
 \[
\|\partial^\alpha [a]\|_{\infty,h} = 
\sup_{i\in\Z^n} \left|h^n\sum_{j\in\Z^n}a_j\partial^\alpha
 \zeta_\varepsilon(ih-jh)\right|\leq
 C\|a\|_{q,h}\varepsilon^{-\frac{n}{q}-|\alpha|}.
\] 
Moreover, if $\zeta\in\cK_c^{k,s}$ and $|\alpha|<s-n$ then for $p\in[1,\infty]$ there is a
constant $C>0$ such that 
\begin{equation}\label{finalestimate}
\|\partial^\alpha[a]\|_{L_\infty(\R^n)} \le C
\|a\|_{q,h}\varepsilon^{-\frac{n}{q}-|\alpha|}. 
\end{equation}
\end{theorem}

\begin{proof}  For $\zeta$ and $p$ as specified in the theorem the first statement follows easily by using H\"older's inequality and Lemma \ref{lemma_core_norm_estimate}, since for each $i\in\Z^n$ we have
\begin{eqnarray*}
\left|h^n\sum_{j\in\Z^n}  
 a_j\partial^\alpha\zeta_\varepsilon(ih-jh)\right|
  &\leq& h^n\sum_{j\in\Z^n}|a_j||\partial^\alpha\zeta_\varepsilon(ih-jh)| \\
  &\leq& \|a\|_{q,h}\|\partial^\alpha\zeta_\varepsilon(ih-.)\|_{p,h}\\
  &\leq& \|a\|_{q,h}\|\partial^\alpha\zeta_\varepsilon\|_{p,h}
\\
& \le &   C\|a\|_{q,h}\varepsilon^{-\frac{n}{q}-|\alpha|}.
\end{eqnarray*}

The second statement follows similarly. For $1\leq p < \infty$ it holds that
\begin{eqnarray*}
\|\partial^\alpha[a]\|_{L_\infty(\R^n)} & = & \sup_{x\in\R^n} \left|h^n\sum_{j\in\Z^n}  
 a_j\partial^\alpha\zeta_\varepsilon(x-jh)\right|\\
& \le & \|a\|_{q,h} 
 \sup_{x\in\R^n}\left(h^n \sum_{j\in\Z^n} |\partial^\alpha \zeta_\varepsilon (x-jh)|^p\right)^{1/p}.
\end{eqnarray*}
Using the compact support of $\zeta_\varepsilon$ shows once again that
the sum is only a sum over at most $C(\varepsilon/h)^n$ terms so that
we can continue with the estimate
\[
\|\partial^\alpha[a]\|_{L_\infty(\R^n)}\le C \|a\|_{q,h} \left(h^n
(\varepsilon/h)^n \varepsilon^{-(n+|\alpha|)p}\right)^{1/p} = C
\|a\|_{q,h}\varepsilon^{-\frac{n}{q}-|\alpha|}.
\]
The remaining case $p=\infty$ is trivial.
\end{proof}

\section{Proof of Convergence}
\label{section_proof}
\setcounter{equation}{0}

In this section, we will prove Theorem \ref{main_result}. As usual, to
simplify the notation, we will suppress the time variable whenever
possible. 

Since our spatial discretisation technique immediately leads to an
infinite system of ordinary differential equations with solutions
 $\{\rho_i^{\varepsilon h}\}$, it is natural to use the discrete
norms defined in Definition \ref{defDiscreteNorm} and to
bound the error 
\[
e_i = \rho_{ih} - \rho_i^{\varepsilon h}, \qquad i\in\Z^n,
\]
using these norms. In this context, it is also quite natural to split the
error into a consistency and a stability error. We have 
\begin{eqnarray}
 \frac{1}{2}\frac{d}{dt}\|e\|_{2,h}^2 &=&
 \frac{1}{2}\frac{d}{dt}h^n\sum_{i\in\Z^n} e_i^2 =  h^n\sum_{i\in\Z^n}
 e_i\dot e_i \nonumber\\ 
 &=& -h^n\sum_{i\in\Z^n} e_i \left(f(t,ih,\rho_{ih},\nabla\rho_{ih}) -
 f(t,ih,[\rho^{\varepsilon h}]_{ih},\nabla[\rho^{\varepsilon
     h}]_{ih})\right) \nonumber\\  
 &=& -h^n\sum_{i\in\Z^n} e_i \left(f(t,ih,\rho_{ih},\nabla\rho_{ih}) -
 f(t,ih,[\rho]_{ih},\nabla[\rho]_{ih})\right) \nonumber\\ 
    &&\mbox{}-h^n\sum_{i\in\Z^n} e_i \left(f(t,ih,[\rho]_{ih},\nabla[\rho]_{ih}) -
 f(t,ih,[\rho^{\varepsilon h}]_{ih},\nabla[\rho^{\varepsilon
     h}]_{ih})\right)\nonumber\\
&=:& -(e_c + e_s).
\end{eqnarray}

The first term in the last expression,
\begin{equation}\label{ec}
e_c=e_c(t):=h^n\sum_{i\in\Z^n} e_i \left(f(t,ih,\rho_{ih},\nabla\rho_{ih}) -
 f(t,ih,[\rho]_{ih},\nabla[\rho]_{ih})\right),
\end{equation}
represents the consistency error of the method while the second term
\begin{equation}\label{es}
e_s=e_s(t):=h^n\sum_{i\in\Z^n} e_i \left(f(t,ih,[\rho]_{ih},\nabla[\rho]_{ih})-f(t,ih,[\rho^{\varepsilon h}]_{ih},\nabla[\rho^{\varepsilon
     h}]_{ih})\right)
\end{equation}
 represents the stability error. We will now bound both errors
 separately, starting with the consistency error.

For the convenience of the reader, we recall our initial
assumptions. We assume that the support of the solution $\rho(t,\cdot)$
is contained in a compact, convex set $\Omega\subseteq\R^n$ for all
$t\in[0,T]$. We have defined the set $M$ and $\widetilde{M}$ in
(\ref{M}) and (\ref{Mtilde}), respectively, to be 
\begin{eqnarray*}
M&:=&\{(t,x,\rho(t,x),\nabla \rho(t,x))\in\R^{2n+2} : t\in[0,T],
x\in\Omega\}\subseteq\R^{2n+2},\\
\widetilde{M}&:=& [0,T]\times\Omega_1\times \Pi. 
\end{eqnarray*}
Here $\Omega_1$ is the convex hull of $\cup_{x\in\Omega} B_1(x)$,
where $B_1(x)$ is the ball of radius $1$ and centre $x$. Moreover,
$\Pi$ is a convex, compact super set of $\{(\rho(t,x),\nabla
\rho(t,x)) : t\in[0,T], x\in\Omega\}\subseteq \R^{n+1}$.

We have also assumed that the defining function $f$ is sufficiently smooth and
has compact support in the compact and convex set
$\widetilde{M}\supseteq M$. As usual, we will consider $f$ to be
defined on all of $\R^n$ with zero value outside $\widetilde{M}$.

The reason for this technical definition is the following one. Suppose
that our kernel $\zeta$ has support in the unit ball and that $\rho(t,\cdot)$
has support in $\Omega$, then, obviously
\[
[\rho] = \sum_{j\in\Z^d} \rho_{jh}(t) \zeta_\varepsilon(\cdot-jh)
\]
has support in $\Omega_1$ for all $0< \varepsilon\le 1$. Moreover, the
convexity of 
$\widetilde{M}$ guarantees that all connecting line segments between
two points in $\widetilde{M}$ are also contained in $\widetilde{M}$.

\begin{proposition}\label{propec}
Let $T>0$. Let $\zeta\in\cK^{k,s}_c$ 
with $k\in\N$ and $s> \ell>n$ for an $\ell\in\N$ and with support in the unit
ball. Let  $\rho$ be the solution of
(\ref{problem_equation}),(\ref{problem_initialValue}). Assume that
$\rho(t,\cdot)\in C^r_c(\R^n)$ with $r\ge \max\{k+1,\ell\}$ and that the
support of $\rho(t,\cdot)$ is 
contained in the compact set $\Omega\subseteq\R^{n}$ for all $t\in
[0,T]$. Suppose further that 
$f\in C^1_c(\widetilde{M})$, with $\widetilde{M}$ defined in
(\ref{Mtilde}). Then,  there is a constant $C>0$ depending  on $f$,
$\zeta$, $\rho$ and $M$ such that the consistency error $e_c$ from 
(\ref{ec}) can be bounded by 
\[
|e_c(t)|\le 
C\left(\varepsilon^k + \frac{h^\ell}{\varepsilon^{\ell+1}}\right)^2
+ \frac{1}{2}\|e(t)\|_{2,h}^2, \qquad t\in[0,T],
\]
for all $0<h\le \varepsilon^{1+\frac{2}{\ell}}$ sufficiently small.
\end{proposition}
\begin{proof} 
Our assumption on $\rho$ immediately shows  $\rho\in
W_\infty^r(\R^n)$. Thus, Theorem \ref{theorem_convolution_quadrature}
gives in particular for each $i\in\Z^n$
\begin{equation}\label{b1}
|\rho_{ih} - [\rho]_{ih}| \le C \left(\varepsilon^k +
\frac{h^\ell}{\varepsilon^{\ell}}\right) 
\end{equation}
and, since $r\ge \max\{k+1,\ell\}$ and $s\ge \ell+1$,  
\begin{equation}\label{b2}
|\nabla \rho_{ih} - \nabla [\rho]_{ih}| \le C\left(\varepsilon^k +
\frac{h^\ell}{\varepsilon^{\ell+1}}\right),
\end{equation}
where the constant $C>0$ depends on $\rho$ and $\zeta$.
Next, let 
\[
I_{\varepsilon h}:=\{i\in\Z^n : ih\in\Omega_\varepsilon\}
\]
and note that the cardinality of $I_{\varepsilon h}$ can be bounded by a constant
times $h^{-n}$  since $\Omega$ is compact. The explanation given above
shows that for $i\not\in I_{\varepsilon h}$ we have $\rho_{ih}=[\rho]_{ih}=0$.

Using the Cauchy-Schwarz inequality and the mean value theorem, there
are $\eta_i\in\R$ on the line segment between $\rho_{ih}$ and
$[\rho]_{ih}$ and $\xi_i\in\R^n$ on the line segment between
$\nabla\rho_{ih}$ and $\nabla[\rho]_{ih}$ such that 
\begin{eqnarray*}
|e_c| & = &  \left|h^n\sum_{i\in\Z^n} e_i
\left(f(t,ih,\rho_{ih},\nabla\rho_{ih}) -
f(t,ih,[\rho]_{ih},\nabla[\rho]_{ih})\right)\right| \\ 
& = & \left|h^n \sum_{i\in I_{\varepsilon h}} e_i
Df(t,ih,\cdot,\cdot)|_{(\eta_i,\xi_i)}\cdot 
(\rho_{ih}-[\rho]_{ih}, \nabla\rho_{ih}-\nabla[\rho]_{ih})\right|\\
 &\leq& |f|_{W_\infty^1(\widetilde{M})}
\|e\|_{2,h}\left(h^n\sum_{i\in I_{\varepsilon h}}  
  |\rho_{ih}-[\rho]_{ih}|^2 + |\nabla\rho_{ih}
  -\nabla[\rho]_{ih}|^2\right)^{\frac{1}{2}} \\  
 &\leq& C(f, \zeta,\rho)\left(\varepsilon^k +
\frac{h^\ell}{\varepsilon^{\ell+1}}\right)\|e\|_{2,h} \\ 
 &\leq& C(f, \zeta,\rho)\left(\varepsilon^k +
\frac{h^\ell}{\varepsilon^{\ell+1}}\right)^2 +
\frac{1}{2}\|e\|_{2,h}^2,
\end{eqnarray*}
where we have also used (\ref{b1}) and (\ref{b2}) as well as the fact that $(\eta_i,\xi_i)\in M_1$ for $\varepsilon>0$ small.
\end{proof}


The next step in our error analysis is to bound the stability error
(\ref{es}). Here, the proof is more demanding and requires a bootstrap
argument, i.e.~we need to make an assumption on the total error to
prove the following bound, which we have to verify later on.

\begin{theorem}\label{theoes}
Let $T>0$. Let $\zeta\in\cK^{k,s}_c$ 
with $s> \ell>n$ for an $\ell\in\N$ being even and with support in the unit
ball. Let  $\rho$ be the solution of
(\ref{problem_equation}),(\ref{problem_initialValue}). Assume that
$\rho(t,\cdot)\in C^r_c(\R^n)$ with $r>k\ge 1$ and that the support of
$\rho(t,\cdot)$ is 
contained in the compact set $\Omega\subseteq\R^{n}$ for all $t\in
[0,T]$. Suppose further that 
$f\in C^2(\widetilde{M})$, where $\widetilde{M}$ is defined in
(\ref{Mtilde}).  Assume that $h\le \varepsilon^{1+2/\ell}$. Assume
finally, that the error satisfies  
\begin{equation}\label{btassumption}
 \|e(t)\|_{2,h} \leq C_1\varepsilon^{2+\frac{n}{2}}, \qquad t\in[0,T],
\end{equation}
with a constant $C_1>0$ independent of $h$ and $\varepsilon$.
Then, there is a constant $C>0$ independent of $\varepsilon>0$ and
$h>0$ such that the  stability error $e_s$ from (\ref{es}) can be bounded by  
\[
|e_s(t)|\le C \|e(t)\|_{2,h}^2, \qquad t\in [0,T],
\]
provided $\varepsilon>0$ is sufficiently small.
\end{theorem}
\begin{proof}
 We start by further splitting the error $e_s$ into
\begin{eqnarray*}
e_s &=&  - h^n\sum_{i\in\Z^n} e_i \left(f(t,ih,[\rho]_{ih},\nabla[\rho]_{ih}) -
f(t,ih,[\rho^{\varepsilon h}]_{ih},\nabla[\rho^{\varepsilon
    h}]_{ih})\right)\\
& =:& e_{s1}+e_{s2} 
\end{eqnarray*}
with 
\begin{eqnarray}  
e_{s1} &:=& -h^n\sum_{i\in\Z^n} e_i
\left(f(t,ih,[\rho]_{ih},\nabla[\rho]_{ih}) -
f(t,ih,[\rho]_{ih},\nabla[\rho^{\varepsilon h}]_{ih})\right) \label{stability_1},\\ 
 e_{s2}&:= & -h^n\sum_{i\in\Z^n} e_i
\left(f(t,ih,[\rho]_{ih},\nabla[\rho^{\varepsilon h}]_{ih}) -
f(t,ih,[\rho^{\varepsilon h}]_{ih},\nabla[\rho^{\varepsilon
  h}]_{ih})\right) \label{stability_2}. 
\end{eqnarray}
Note that our assumption on the support of $f$ means that all the sums are
actually finite sums, summing at most over those indices $i\in\Z^n$
with $ih\in \Omega_1$.

We will now bound each term separately, starting with
(\ref{stability_1}). Using once again the mean value theorem, this
time only with respect to the last argument of $f$, yields positions
$\xi_i\in\R^n$ on the line segment connecting $\nabla[\rho]_{ih}$ and
$\nabla[\rho^{\varepsilon 
    h}]_{ih}$ such that
\begin{eqnarray}
e_{s1} & = & 
 -h^n\sum_{i\in\Z^n} e_i \left(f(t,ih,[\rho]_{ih},\nabla[\rho]_{ih}) -
 f(t,ih,[\rho]_{ih},\nabla[\rho^{\varepsilon h}]_{ih})\right) \nonumber\\ 
 &=& -h^n\sum_{i\in\Z^n} e_i D
 f(t,ih,[\rho]_{ih},\cdot)|_{\xi_i}\cdot\left(\nabla[\rho]_{ih} -
 \nabla[\rho^{\varepsilon h}]_{ih}\right) \nonumber\\ 
 &=& -h^n\sum_{i\in\Z^n} e_i D
 f(t,ih,[\rho]_{ih},\cdot)|_{\xi_i}\cdot\left(h^n\sum_{j\in\Z^n}(\rho_j-\rho_j^{\varepsilon
   h})\nabla\zeta_\varepsilon(ih-jh)\right) \nonumber\\ 
 &=& -h^{2n}\sum_{i,j\in\Z^n} e_ie_j D
 f(t,ih,[\rho]_{ih},\cdot)|_{\xi_i}\cdot\nabla\zeta_\varepsilon(ih-jh). \label{es1-1}
\end{eqnarray}

Next, note that the assumption that the kernel $\zeta$ is an
even function means in particular that 
$\nabla\zeta_\varepsilon(-\cdot)=-\nabla\zeta_\varepsilon(\cdot)$ and
hence $\nabla\zeta_\varepsilon(0)=0$. To make use of this
anti-symmetry, we partition $\Z^n\times \Z^n$ 
disjointly into $\Z^n\times\Z^n=\Lambda\cup\overline{\Lambda}\cup
\{(i,i) : i\in\Z^n\}$ where $\Lambda,
\overline{\Lambda}\subseteq\Z^n\times\Z^n$ are such that for every
$i\neq j$ we have $(i,j)\in\Lambda$ if and only if
$(j,i)\in\overline\Lambda$.  

Since $\nabla \zeta_{\varepsilon}(0)=0$, we see that we can ignore all
entries in the sum (\ref{es1-1}) corresponding to indices $(i,i)$, $i\in\Z^n$.
Hence, we can continue to rewrite $e_{s1}$ by
\begin{eqnarray}
e_{s1} & =&  -h^{2n}\sum_{i,j\in\Z^n} e_ie_j D
 f(t,ih,[\rho]_{ih},\cdot)|_{\xi_i}\cdot\nabla\zeta_\varepsilon(ih-jh)\nonumber\\
 &=& -h^{2n}\left(\sum_{(i,j)\in\Lambda} +\sum_{(i,j)\in\overline\Lambda}\right) e_ie_j D
f(t,ih,[\rho]_{ih},\cdot)|_{\xi_i}\cdot\nabla\zeta_\varepsilon(ih-jh) \nonumber\\ 
& =& -h^{2n}\sum_{(i,j)\in\Lambda}  e_ie_j D
f(t,ih,[\rho]_{ih},\cdot)|_{\xi_i}\cdot\nabla\zeta_\varepsilon(ih-jh) \nonumber\\ 
&  &\mbox{} -
h^{2n}\sum_{(j,i)\in\overline\Lambda} e_je_i D
f(t,jh,[\rho]_{jh},\cdot)|_{\xi_j}\cdot\nabla\zeta_\varepsilon(jh-ih) \nonumber\\ 
& = & -h^{2n}\sum_{(i,j)\in\Lambda}  e_ie_j D
f(t,ih,[\rho]_{ih},\cdot)|_{\xi_i}\cdot\nabla\zeta_\varepsilon(ih-jh) \nonumber\\ 
 & &\mbox{}-
h^{2n}\sum_{(i,j)\in\Lambda} e_je_i D
f(t,jh,[\rho]_{jh},\cdot)|_{\xi_j}\cdot\left(-\nabla\zeta_\varepsilon(ih-jh)\right)
\nonumber\\ 
 &=& -h^{2n}\sum_{(i,j)\in\Lambda} e_ie_j \left(D
f(t,ih,[\rho]_{ih},\cdot)|_{\xi_i} - D
f(t,jh,[\rho]_{jh},\cdot)|_{\xi_j}\right)\cdot\nabla\zeta_\varepsilon(ih-jh)\nonumber\\ 
 &=& -h^{2n}\sum_{(i,j)\in\Lambda} e_ie_j
\Delta_{ij}\cdot\nabla\zeta_\varepsilon(ih-jh),\label{es1-2}
\end{eqnarray}
where we have introduced the notation
\[
\Delta_{ij}:= D
f(t,ih,[\rho]_{ih},\cdot)|_{\xi_i} - D
f(t,jh,[\rho]_{jh},\cdot)|_{\xi_j},
\]
which obviously satisfies $\Delta_{ij}=-\Delta_{ji}$. Moreover, we can
simply set $\Delta_{ij}=0$ for $(i,j)\in\Lambda$ satisfying $|ih-jh|\ge
\varepsilon$, since in this situation our compactly supported kernel
makes sure that the corresponding term in (\ref{es1-2}) is already
zero.

Having this in mind, we can now derive the following bound on
$e_{s1}$: 
\begin{eqnarray}
|e_{s1}| 
 &\leq& h^{2n}\sum_{(i,j)\in\Lambda} \frac{1}{2}(e_i^2+e_j^2)
|\Delta_{ij}|\,|\nabla\zeta_\varepsilon(ih-jh)|\nonumber\\ 
 &\leq& h^{2n}\sum_{i,j\in\Z^n} e_i^2
|\Delta_{ij}|\,|\nabla\zeta_\varepsilon(ih-jh)|\nonumber\\ 
 &=& h^{n}\sum_{i\in\Z^n} e_i^2 \left(h^n\sum_{j\in\Z^n}|\Delta_{ij}|\,|
\nabla\zeta_\varepsilon(ih-jh)|\right)\nonumber\\ 
 &\leq& \|e\|_{2,h}^2 \max_{i\in\Z^n}
h^n\sum_{j\in\Z^n}|\Delta_{ij}|\,|\nabla\zeta_\varepsilon(ih-jh)|. \label{stability_1_beforeDifferenceEstimation} 
\end{eqnarray}
To bound this further, we need a thorough estimate for the term
$|\Delta_{ij}|$, which, in particular, has to compensate the $1/\varepsilon$
factor coming from $|\nabla\zeta_\varepsilon(ih-jh)|$. We have
\begin{eqnarray}
|\Delta_{ij}| &=& \left| D f(t,ih,[\rho]_{ih},\xi_i)
 - D f(t,jh,[\rho]_{jh},\xi_j)\right| \nonumber\\
&\leq& \|f\|_{W_{\infty}^2(\widetilde{M})}\left(|ih-jh|+|[\rho]_{ih}-[\rho]_{jh}|+ 
|\xi_i-\xi_j|\right) \nonumber\\
&\leq& \|f\|_{W_{\infty}^2(\widetilde{M})}\left(\varepsilon+|[\rho]_{ih}-[\rho]_{jh}|+ 
|\xi_i-\xi_j|\right) \label{stability_1_differenceEstimation},
\end{eqnarray}
where we used the fact that we only have to consider indices $(i,j)$ with
$|ih-jh|\le \varepsilon$. To continue our estimate, we use
\begin{equation}\label{est2}
 |\xi_i-\xi_j| \leq |\xi_i-
 \nabla[\rho]_{ih}|+|\nabla[\rho]_{ih}-\nabla[\rho]_{jh}|+|\nabla[\rho]_{jh}
 -\xi_j|. 
\end{equation}
Here, the second term on the right-hand side, as well as the term
$|[\rho]_{ih}-[\rho]_{jh}|$ in 
\eqref{stability_1_differenceEstimation} can be bounded as
follows. For $\alpha\in\N_0^n$ with $|\alpha|=0$ or $|\alpha|=1$ we
have, using Theorem \ref{theorem_convolution_quadrature},
\begin{eqnarray*}
|\partial^\alpha[\rho]_{ih}-\partial^\alpha[\rho]_{jh} | & \le & 
|\partial^\alpha[\rho]_{ih}-\partial^\alpha\rho_{ih}| +
|\partial^\alpha\rho_{ih}-\partial^\alpha\rho_{jh}| +
|\partial^\alpha\rho_{jh}-\partial^\alpha[\rho]_{jh}|\\
& \le & C(\rho)\left(
\varepsilon^k+\frac{h^\ell}{\varepsilon^{\ell+|\alpha|}}\right)+ 
\|\rho\|_{W_\infty^1(\Omega)} |ih-jh|\\
& \le& C(\rho)\varepsilon,
\end{eqnarray*}
since $k\ge 1$ and $h\le \varepsilon^{1+2/\ell}$.
For the first and last term in (\ref{est2}) we use the fact that $\xi_i$
lies on the line segment connecting $\nabla[\rho]_{ih}$ and
$\nabla[\rho^{\varepsilon h}]_{ih}$ such that we can derive the estimate
\begin{eqnarray*}
|\Delta_{ij}| &\leq& \|f\|_{W_\infty^2(\widetilde{M})}\left(\varepsilon+|[\rho]_{ih}-[\rho]_{jh}|+
 |\xi_i-\xi_j|\right) \\
&\leq &C(f,\rho)\left(\varepsilon +
 \|\nabla[\rho]_{.}-\nabla[\rho^{\varepsilon h}]_{.}\|_{\infty,h}\right).
\end{eqnarray*}

If we plug this into \eqref{stability_1_beforeDifferenceEstimation}
and use Lemma \ref{lemma_core_norm_estimate} and Theorem
\ref{theorem_convolution_quadrature_stability} (which is possible since we
have $s>n+1$) we find
\begin{eqnarray}
|e_{s1}| 
 &\leq& \|e\|_{2,h}^2C(f,\rho)\left(\varepsilon +
\|\nabla[\rho]-\nabla[\rho^{\varepsilon
  h}]\|_{\infty,h}\right) \|\nabla\zeta_\varepsilon\|_{1,h}
\nonumber\\ 
 &\leq& \|e\|_{2,h}^2C(f,\rho)\left(\varepsilon +
C\varepsilon^{-1-\frac{n}{2}}\|\rho-\rho^{\varepsilon
  h}\|_{2,h}\right)\varepsilon^{-1}\nonumber\\ 
 &\leq& \|e\|_{2,h}^2C(f,\rho)\left(\varepsilon +
C\varepsilon^{-1-\frac{n}{2}}\|e\|_{2,h}\right)\varepsilon^{-1}. \nonumber\\ 
& = &
C(f,\rho)\left(1+C\varepsilon^{-2-\frac{n}{2}}\|e\|_{2,h}\right)\|e\|_{2,h}^2. 
\label{stability_1_final} 
\end{eqnarray}
Obviously, if condition (\ref{btassumption}) holds, i.e. if we
have $ \|e\|_{2,h} \leq C_1\varepsilon^{2+\frac{n}{2}}$, then it
immediately follows from \eqref{stability_1_final} that we also have
\[
|e_{s1}|\le  C(f,\rho)\|e\|_{2,h}^2.
\]

Finally, we need to bound the second part \eqref{stability_2} of the
stability error. As before, we will use the mean value theorem and
denote the intermediate positions by $\xi_i$ again. This time, we have
\begin{eqnarray*}
e_{s2} & = & 
 -h^n\sum_{i\in\Z^n} e_i
 \left(f(t,ih,[\rho]_{ih},\nabla[\rho^{\varepsilon h}]_{ih}) -
 f(t,ih,[\rho^{\varepsilon h}]_{ih},\nabla[\rho^{\varepsilon
   h}]_{ih})\right) \nonumber\\ 
 &=& -h^n\sum_{i\in\Z^n} e_i Df(t,ih,\cdot,\nabla[\rho^{\varepsilon
   h}]_{ih})|_{\xi_i}\left([\rho]_{ih}-[\rho^{\varepsilon h}]_{ih}\right)
 \nonumber\\ 
 &=& -h^n\sum_{i\in\Z^n} e_i Df(t,ih,\cdot,\nabla[\rho^{\varepsilon
   h}]_{ih})|_{\xi_i}|\left(h^n\sum_{j\in\Z^n}(\rho_j-\rho_j^{\varepsilon
   h})\zeta_\varepsilon(ih-jh)\right) \nonumber\\  
 &=& -h^{2n}\sum_{i,j\in\Z^n} e_ie_j Df(t,ih,\cdot,\nabla[\rho^{\varepsilon
   h}]_{ih})|_{\xi_i}|\zeta_\varepsilon(ih-jh).
\end{eqnarray*}
This gives the bound
\begin{eqnarray}
|e_{s2}| 
 &\leq& h^{2n}\|f\|_{W_\infty^1(\widetilde{M})} \sum_{i,j\in\Z^n} |e_ie_j|
|\zeta_\varepsilon(ih-jh)|\nonumber\\ 
 &\leq& C(f)h^{n}\sum_{i\in\Z} e_i^2\;
h^n\sum_{j\in\Z^n}|\zeta_\varepsilon(ih-jh)|\nonumber\\ 
& = & C(f) \left( h^n\sum_{i\in\Z^n} e_i^2
  \right)\left(h^n\sum_{j\in\Z^n}\zeta_\varepsilon(jh)\right) \nonumber\\
 &\leq& C(f)\|e\|_{2,h}^2,
\end{eqnarray}
which, together with (\ref{stability_1_final}), proves the statement of
the theorem.
\end{proof}

\noindent {\bf Proof of Theorem \ref{main_result}:}
Taking the results of Proposition \ref{propec} and Theorem
\ref{theoes} together and assuming (\ref{btassumption}), we see that
\[
 \frac{1}{2}\frac{d}{dt}\|e\|_{2,h}^2 \leq C(f,\rho)\left(\varepsilon^k + \frac{h^\ell}{\varepsilon^{\ell+1}}\right)^2 + C(f,\rho)\|e\|_{2,h}^2.
\]
Hence, applying Gronwall's inequality to this, yields
\begin{equation}\label{eestimate}
 \|e\|_{2,h} \leq C(f,\rho, T)\left(\varepsilon^k +
 \frac{h^\ell}{\varepsilon^{\ell+1}}\right) 
\end{equation}
uniformly for all $t\in[0,T]$. With this, we can justify the
assumption (\ref{btassumption}) since we have 
\[
\varepsilon^k +
 \frac{h^\ell}{\varepsilon^{\ell+1}} \le C
 \varepsilon^{2+\frac{n}{2}},
\]
provided that $k\ge 2+n/2$ and $h\le
\varepsilon^{1+(3+\frac{n}{2})/\ell}$.

Next, we split the $L_\infty(\R^n)$ error as follows
\begin{equation}\label{finalsplit}
\|\partial^\alpha \rho - \partial^\alpha[\rho^{\varepsilon
    h}]\|_{L_\infty(\R^n)} 
\le \|\partial^\alpha \rho - \partial^\alpha[\rho]\|_{L_\infty(\R^n)}
 +
\|\partial^\alpha [\rho] - \partial^\alpha[\rho^{\varepsilon
    h}]\|_{L_\infty(\R^n)}. 
\end{equation}
We can bound the first term on the right-hand side using Theorem
\ref{theorem_convolution_quadrature} by 
\[
 \|\partial^\alpha \rho - \partial^\alpha[\rho]\|_{L_\infty(\R^n)}
\le C(\rho) \left(\varepsilon^k + \frac{h^\ell}{\varepsilon^{\ell+|\alpha|}}\right)
\]
Since this expression will be dominated by the second term in
(\ref{finalsplit}), we can ignore it. The second term is finally
bounded by 
\begin{eqnarray*}
\|\partial^\alpha [\rho] - \partial^\alpha[\rho^{\varepsilon
    h}]\|_{L_\infty(\R^n)} & \le &
C \varepsilon^{-\frac{n}{2}-|\alpha|} \|\rho-\rho^{\varepsilon h}\|_{2,h} \\
& \le & C \varepsilon^{-\frac{n}{2}-|\alpha|}
\left(\varepsilon^k+\frac{h^\ell}{\varepsilon^{\ell+1}}\right) 
\end{eqnarray*}
using (\ref{finalestimate}) from Theorem
\ref{theorem_convolution_quadrature_stability} and (\ref{eestimate}).
\hfill\qed

\section{Construction of High Order Kernels}

A key ingredient of the method is the availability of high-order
kernels. There are some ways of constructing such kernels, and here, we will 
follow ideas from \cite{Beale-Majda-85-1}, see also
\cite{Majda-Bertozzi-02-1}, to construct compactly supported kernels
of any prescribed order.

To this end, we will employ radial kernels,
i.e. kernels of the form $\kernel(x) = \odkernel(|x|)$, $x\in\R^n$, with a
univariate function $\odkernel:[0,\infty)\to\R$.
For such radial kernels we can easily rewrite the conditions from
Definition \ref{definition_kernel}:
\begin{eqnarray*}
\int_{\R^n} x^\alpha \kernel(x) dx &=&  \int_0^\infty \int_{|x|=1}
(xr)^\alpha \kernel(rx)r^{n-1} dS(x) dr \\
& = &\int_{|x|=1} x^\alpha dS(x) \int_0^\infty r^{n-1+|\alpha|} \odkernel(r)
dr.
\end{eqnarray*}

Obviously the first integral over the unit sphere $S^{n-1}$ in $\R^n$
vanishes if $|\alpha|$ is odd. Thus the conditions of Definition
\ref{definition_kernel} can be 
rewritten as follows. The kernel $\kernel=\odkernel(|\cdot|)$ is of order $k=2\ell$ if
it  satisfies the following three conditions 
\begin{eqnarray}
\int_0^\infty \odkernel(r)r^{n-1} dr & = & \frac{1}{\omega_{n-1}},\label{c1} \\
\int_0^\infty \odkernel(r)r^{n+2j-1}dr & = & 0, \qquad 1\le j\le \ell-1\label{c2}\\
\int_0^\infty \odkernel(r)r^{n-1+k}dr &<&\infty\label{c3},
\end{eqnarray}
where $\omega_{n-1}$ denotes the surface area of the unit sphere
$S^{n-1}$ in $\R^n$.

We will use this to construct such kernels. To this end assume that we
have fixed, pairwise distinct 
values $a_j> 0 $ for $0\le j\le \ell-1$ and an even, continuous,
non-negative univariate function $\odkernelbase:\R\to\R$ with compact support and 
$\|\odkernelbase\|_{L_1(\R)}=1$. 

Then, we want to pick real numbers $\lambda_0,\ldots,\lambda_{\ell-1}$
such that 
\begin{equation}\label{ridge}
\odkernel(r) = \sum_{j=0}^{\ell-1} \lambda_j \odkernelbase(r/a_j), 
\end{equation}
defines a kernel of order $k=2\ell$.

Since $\odkernel$ also has compact support and is continuous, condition
(\ref{c3}) is automatically satisfied. Conditions (\ref{c1}) and
(\ref{c2}) can be summarised as follows. For $0\le i\le
\ell-1$ we need that
\begin{eqnarray*}
\frac{1}{\omega_{n-1}}\delta_{0i} & = & \sum_{j=0}^{\ell-1}\lambda_j
\int_0^\infty \odkernelbase(r/a_j)r^{n-1+2i} dr \\
& = & \sum_{j=0}^{\ell-1}\lambda_ja_j^{n+2i}\int_0^\infty \odkernelbase(s)
s^{n-1+2i} ds.
\end{eqnarray*}
Noting that the integral on the right-hand side is independent of the
summation index $j$ and that the left-hand side is zero except
for $i=0$ we see that we can rewrite this system as 
\[
\delta_{0i}=\sum_{j=0}^{\ell-1}\widetilde{\lambda}_j a_j^{2i}, \qquad
0\le i\le \ell-1,
\]
where we have set $\widetilde{\lambda}_j := \lambda_j a_j^n$. This
means that the solution vector $\widetilde{\lambda}\in\R^\ell$ is the
first column of the inverse of the matrix $A=(a_j^{2i})$, which simply
is the transpose of a Vandermonde matrix in
$a_0^2,\ldots,a_{\ell-1}^2$. This guarantees solvability. To be more
precise, we have the following result.

\begin{proposition}\label{constr}
Let $a_j>0 $, $0\le j\le \ell-1$, be pairwise distinct and let $\odkernelbase:\R\to\R$ be
non-negative, even, continuous, with compact support and satisfying
$\|\odkernelbase\|_{L_1(\R)}=1$. Then, there is exactly one radial kernel
$\kernel(x)=\odkernel(|x|)$, $x\in\R^n$, of order $k=2\ell$ with $\odkernel$ of the
form (\ref{ridge}). The coefficients are given by
\begin{eqnarray*}
\lambda_j &=&  \frac{(-1)^ja_0^2\cdots a_{j-1}^2a_{j+1}^2\cdots
  a_{\ell-1}^2}{a_j^n(a_j^2-a_0^2)\cdots 
  (a_j^2-a_{j-1}^2)(a_{j+1}^2-a_j^2)\cdots (a_{\ell-1}^2-a_j^2)}\\
& = & \frac{1}{a_j^n}\prod_{\substack{i=0\\i\ne j}}^{\ell-1}\frac{a_i^2}{a_i^2-a_j^2}, \qquad
0\le j\le \ell-1.
\end{eqnarray*}
\end{proposition}
\begin{proof}
Using the notation from the paragraphs above, 
Cramer's rule shows that $\widetilde{\lambda_j}=\det A_j/\det A$,
where $A_j$ is the matrix resulting from $A$ by replacing the $j$-th
column with the first unit vector. Since, $A$ is the transpose of a
Vandermonde matrix, its determinant is given by
\[
\det (A) = \prod_{0\le \mu<\nu\le \ell-1} (a_\nu^2-a_\mu^2).
\]
Moreover, in our situation the determinant of $A_j$ is easily
determined. Setting $b_j=a_j^2$, we have

\begin{eqnarray*}
\det(A_j)& = & \det \begin{pmatrix}
1  & \hdots&   1      &  1 &1  &   \hdots     & 1        \\
b_0& \hdots &  b_{j-1} & 0 & b_{j+1} & \hdots & b_{\ell-1} \\
b_0^2& \hdots &  b_{j-1}^2 & 0 & b_{j+1}^2 & \hdots & b_{\ell-1}^2 \\
\vdots&       &\vdots    & \vdots&\vdots&        &\vdots\\
b_0^{\ell-1}& \hdots &  b_{j-1}^{\ell-1} & 0 & b_{j+1}^{\ell-1} &
\hdots & b_{\ell-1}^{\ell-1} \\ 
\end{pmatrix}\\
& = & (-1)^j
\det \begin{pmatrix}
b_0&   \hdots &  b_{j-1}   &  b_{j+1}   & \hdots & b_{\ell-1} \\
b_0^2& \hdots &  b_{j-1}^2 &  b_{j+1}^2 & \hdots & b_{\ell-1}^2 \\
\vdots&      &\vdots     &\vdots     &        &\vdots\\
b_0^{\ell-1}& \hdots &  b_{j-1}^{\ell-1}  & b_{j+1}^{\ell-1} &
\hdots & b_{\ell-1}^{\ell-1} \\ 
\end{pmatrix}\\
& = & (-1)^j\left(\prod_{\substack{i=0 \\i\ne j}}^{\ell-1}b_i\right)
\det \begin{pmatrix}
1  & \hdots&     1    & 1  &    \hdots    & 1        \\
b_0& \hdots &  b_{j-1} & b_{j+1} & \hdots & b_{\ell-1} \\
b_0^2& \hdots &  b_{j-1}^2 & b_{j+1}^2 & \hdots & b_{\ell-1}^2 \\
\vdots&       &\vdots    &\vdots&        &\vdots\\
b_0^{\ell-2}& \hdots&b_{j-1}^{\ell-2}  & b_{j+1}^{\ell-2} &
\hdots & b_{\ell-1}^{\ell-2} \\ 
\end{pmatrix}\\
& = & (-1)^j\left(\prod_{\substack{i=0\\i\ne j}}^{\ell-1}b_i\right)\left(
\prod_{\substack{0\le \mu<\nu\le \ell-1\\\mu,\nu\ne j}}(b_\nu-b_\mu)\right)\\
\end{eqnarray*}
which gives the stated form.
\end{proof}

This simple proof also follows from more general results on
Vandermonde matrices, see for example
\cite{Bjoerck-Pereyra-70-1,Eisinberg-Fedele-06-1} and the references
therein. 

Here, we propose to use the following compactly supported kernels. We
start with radial kernels 
$\odkernelbase(r)=\odkernelbase_{n,k}(r)=c_{n,k}(1-r)_+^{\ell(n,k)}p_{n,k}(r)$ from
\cite{Wendland-95-1, Wendland-05-1}, see also Table
\ref{tab1}. These kernels are known to have smoothness $C^{2k}(\R^n)$ and are as
non-negative, radial kernels of order $2$. The cases $n=2$ and $n=3$ only differ 
in the constant $c_{n,k}$, which has to
be chosen such that the kernels satisfy (\ref{c1}).

\begin{table}
\begin{center}
\begin{tabular} {|l|l|l|}\hline
$n = 1$ & $\odkernelbase_{1,0}(r) = (1-r)_+   $& $C^0$ \\[0.5ex]
        & $\odkernelbase_{1,1}(r) = \frac{5}{4}(1-r)_+^3(3r+1)$ & $C^2$ \\[0.5ex]
        & $\odkernelbase_{1,2}(r) = \frac{3}{2}(1-r)_+^5(8r^2+5r+1)$ & $C^4$ \\[0.5ex]
        & $\odkernelbase_{1,3}(r) = \frac{55}{32}(1-r)_+^7(21r^3+19r^2+7r+1)$ & $C^6$ \\[0.2ex]\hline
$n = 2$ & $\odkernelbase_{2,0}(r) = \frac{6}{\pi}(1-r)^2_+$ & $ C^0$ \\[0.5ex]
        & $\odkernelbase_{2,1}(r) = \frac{7}{\pi}(1-r)^4_+(4r+1)$ & $C^2$ \\[0.5ex]
        & $\odkernelbase_{2,2}(r) = \frac{3}{\pi}(1-r)^6_+(35r^2+18r+3)$ &$C^4$ \\[0.5ex]
        & $\odkernelbase_{2,3}(r) = \frac{78}{7\pi}(1-r)_+^8(32r^3+25r^2+8r+1)$ &$C^6$\\[0.2ex]\hline
$n = 3$ & $\odkernelbase_{3,0}(r) = \frac{15}{2\pi}(1-r)^2_+$ & $ C^0$ \\[0.5ex]
        & $\odkernelbase_{3,1}(r) = \frac{21}{2\pi}(1-r)^4_+(4r+1)$ & $C^2$ \\[0.5ex]
        & $\odkernelbase_{3,2}(r) = \frac{165}{32\pi}(1-r)^6_+(35r^2+18r+3)$ &$C^4$ \\[0.5ex]
        & $\odkernelbase_{3,3}(r) = 
        \frac{1365}{64\pi}(1-r)_+^8(32r^3+25r^2+8r+1)$ &$C^6$\\[0.2ex]\hline
\end{tabular}
\caption {Table of functions\label{tab1}}
\end{center}
\end{table}

For higher order kernels we employ the construction from Proposition
\ref{constr}. Here, we are free to choose the parameters $a_j$. It
might be interesting to discuss the optimal choice of these parameters with
respect to, for example, stability of the evaluation of the
kernels. Here, however, we made the choice given in Table \ref{tab2},
which also contains the corresponding weights $\lambda_j$.

\begin{table}
\begin{center}
\begin{tabular}{|c|ccc|ccc|}\hline
Order & $a_0$ & $a_1$ & $a_2$  & $\lambda_0$ & $\lambda_1$ & $\lambda_2$\\\hline
4     & 1.0  & 4/5  &        & -16/9       &125/26 &\\ \hline
6     & 1.0   & 4/5  & 3/5    & 1.0         &-125/28 & 125/21\\\hline
\end{tabular}
\caption{Possible coefficients and weights for fourth and sixth order
  kernels.\label{tab2}}.
\end{center}
\end{table}

Note that these coefficients can be used for {\em any} radial kernel
with compact support. In the following, we will denote a kernel of
smoothness $s$ and order $k$ by $\odkernel=\odkernel^{k,s}$, i.e.~for
$\kernel^{k,s}(\cdot)=\odkernel^{k,s}(|\cdot|)$ we have 
$\kernel^{k,s}\in \cK^{k,s}$. Hence, typical examples are
\begin{eqnarray*}
\odkernel^{2,2}(r) &=& \odkernelbase_{n,1}(r)\\
\odkernel^{4,4}(r) &=& -\frac{16}{9}\odkernelbase_{n,2}(r) +\frac{125}{26}\odkernelbase_{n,2}(5r/4)\\
\odkernel^{4,6}(r) &=& -\frac{16}{9}\odkernelbase_{n,3}(r) +\frac{125}{26}\odkernelbase_{n,3}(5r/4),\\
\odkernel^{6,6}(r) &=& \odkernelbase_{n,3}(r) -\frac{125}{28}\odkernelbase_{n,3}(5r/4) +
\frac{125}{21}\odkernelbase_{3,3}(5r/3).
\end{eqnarray*}
Some of these kernels are also depicted in Figure \ref{fig1}.

\begin{figure}[htb]
\centering
\resizebox{0.49\textwidth}{!}{\input{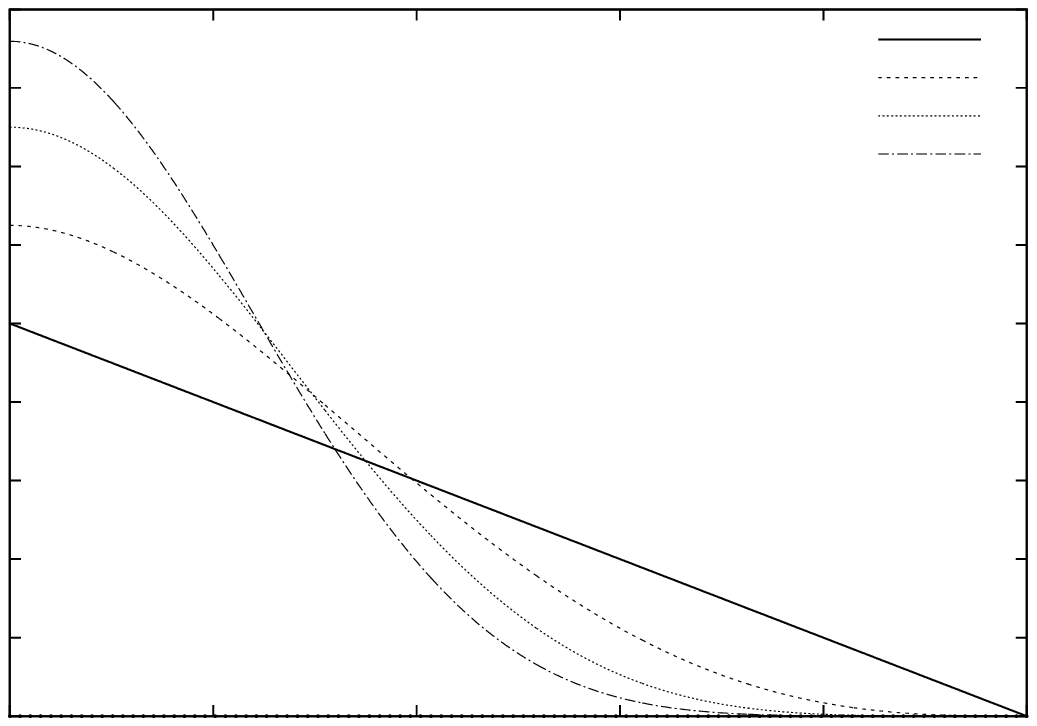}}
\resizebox{0.49\textwidth}{!}{\input{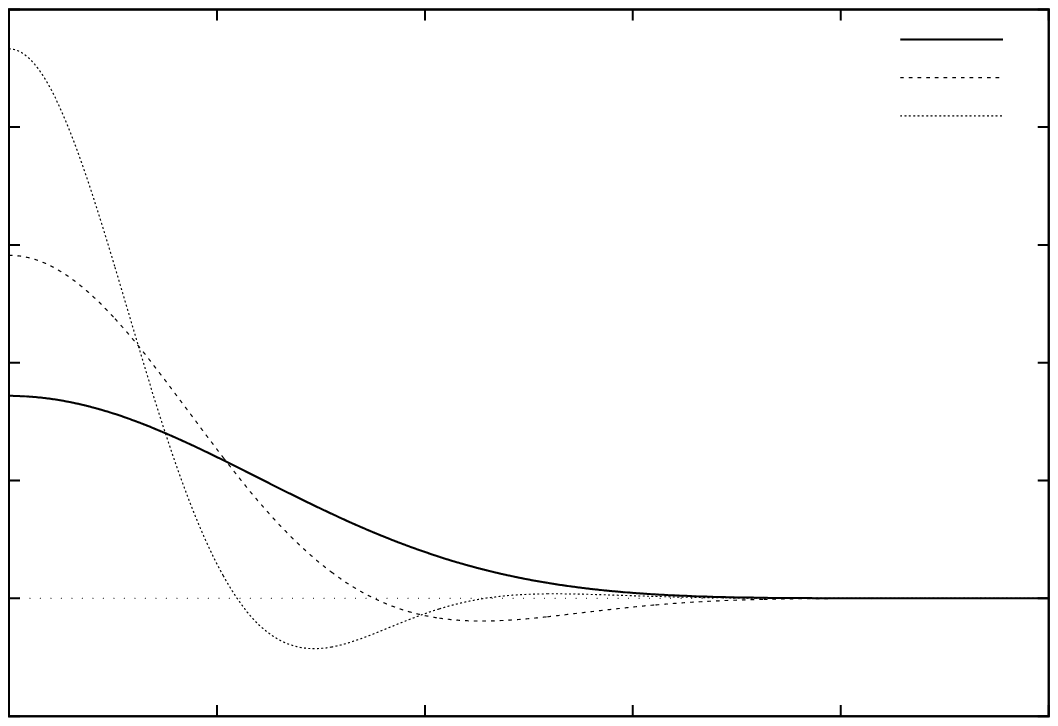}}
\vspace{-0.2cm}
\caption{Some of the kernels.}
\label{fig1}
\end{figure}

\section{Numerical Example}

We are now going to test our method by looking at a numerical example
in one space dimension. To be more precise, we consider the following
Burgers equation. For a given  $T>0$ and $\solution\in C^r(\R)$ with compact
support, we are looking for the solution $\rho$ of 
\begin{align*}
 \partial_t \rho - \rho\partial_x\rho & = [\left(1-\solution\right)\solution'](x+t),  && \text{on } (0,T] \times \R,\\
             \rho|_{t=0}                 & = \solution, && \text{on } \R.
\end{align*}
 Thus, in our initial setting, the defining function $f$ is given by
\[
f(t,x,\rho,\partial_x\rho) =  \rho(t,x)\partial_x\rho(t,x) +
\left[\left(1-\solution\right)\solution'\right](x+t),\qquad (t,x)\in[0,T]\times\R. 
\]

It is easy to see that $\rho(t,x)=\solution(x+t)$ solves the problem above. 
Hence, we can control the smoothness of the
solution. Moreover, if we pick the initial data $\solution$ with compact
support in the one-dimensional {\em ball} $[-\delta,\delta]$ then the
support of $\rho(t,\cdot)$ is in $[-t-\delta,-t+\delta]$ 
and hence, for a fixed
$T>0$, the support of $\rho(t,\cdot)$, $t\in [0,T]$, is a subset of the 
interval $\Omega=[-\delta-T,\delta]$. Thus, the assumption on the
solution from Theorem \ref{main_result} are easily satisfied and we can try to
verify the convergence orders claimed  therein. To this end, we
carried out the following two test series.

\begin{enumerate}
 \item[(A)] In the first series we have chosen the kernel and the
   solution such that Theorem \ref{main_result} is applicable. To be
   more precise,  we used  $\kernel=\odkernel^{4,4}(|\cdot|)\in
   \cK^{4,4}$ as the kernel for our scheme.  The kernel has been built
   as described in the last section using the underlying function
   $\odkernelbase_{1,2}$  and the coefficients given in Table \ref{tab2}.  

The initial data was given by 
$\solution=\delta^{-1}\tilde\odkernel^{4,4}(\cdot/\delta)$ with
$\delta=0.5$. To avoid any unwanted, positive side effects,
which might result from chosing the same function as the approximation
kernel and the initial data, we have used $\odkernelbase_{2,2}$ as the
underlying kernel to built $\solution$. Of course,  we have chosen
$c_{2,2}$ as  $\frac{9}{16}$ for the kernel so that
$\|\odkernelbase_{2,2}\|_{L_1(\R)}=1$ is satisfied.

\item [(B)] The purpose of the second series was to investigate whether 
our scheme also shows convergence in situations with less regularity
and order of the kernel than required by our main result.

To this end, we used 
$\kernel=\odkernel^{2,2}(|\cdot|)=\odkernelbase_{1,1}(|\cdot|)\in \cK^{2,2}$ 
as the kernel and  $\solution=\delta^{-1}\odkernelbase_{2,1}(\cdot/\delta)$ 
as the initial data with $c_{2,1}=\frac{3}{2}$ to achieve 
$\|\odkernelbase_{2,1}\|_{L_1(\R)}=1$ again.
\end{enumerate}

In both cases we have chosen $\delta=T=0.5$, meaning 
$\Omega=[-1.0,0.5]$. Our computation was then restricted to those data sites 
in $\Omega_1=[-1.25,0.75]$, which, for computational reasons, we have chosen 
slightly smaller than the one we have used in Theorem \ref{main_result}.

The computations have been done for 
various values of $\varepsilon$ and $h$, both of the form $2^{-\nu}$. For the
time discretisation we have chosen an explicit Runge-Kutta method of
order $4$. The numerical results indicated that there is, as expected,
a CFL condition. The theoretical analysis of the time discretisation
will be subject of a subsequent paper. Here, we simply have chosen the
time discretisation sufficiently small such that its error was
negligible.

\begin{table}[tbp]
\centering
\begin{tabular}{|c|c|c|c|c|c|c|c|c|}
\cline{2-9}
\multicolumn{1}{c|}{} & \multicolumn{8}{c|}{$\nu_\varepsilon$} \\
\hline
$\nu_h$ &   -6 & -7 & -8 & -9 & -10 & -11 & -12 & -13 \\ \hline
-9 &  4.83e-3 &  &  &  &  &  &  &  \\ 
-10 & 1.47e-3 & 7.09e-3 &  &  &  &  &  &  \\ 
-11 & 1.03e-4 & 2.27e-3 & 1.05e-2 &  &  &  &  &  \\ 
-12 & 1.18e-4 & 3.39e-5 & 3.31e-3 & 1.62e-2 &  &  &  &  \\ 
-13 & 1.18e-4 & 1.71e-5 & 6.72e-5 & 5.28e-3 & 2.47e-2 &  &  &  \\ 
-14 & 1.18e-4 & 1.55e-5 & 3.90e-6 & 1.11e-4 & 8.53e-3 & 3.48e-2 &  &  \\ 
-15 & 1.18e-4 & 1.54e-5 & 1.32e-6 & 4.60e-6 & 1.86e-4 & 1.33e-2 & 4.26e-2 &  \\ 
-16 &  & 1.54e-5 & 1.27e-6 & 2.29e-7 & 7.52e-6 & 3.19e-4 & 1.89e-2 & 4.67e-2 \\ 
-17 &  &  & 1.27e-6 & 1.49e-7 & 1.47e-7 & 1.29e-5 & 5.42e-4 & 2.33e-2 \\ 
-18 &  &  &  & 1.53e-7 & 1.57e-8 & 2.27e-7 & 2.21e-5 & 9.09e-4 \\ \hline
\multicolumn{1}{c}{}
\end{tabular}
\caption{Discrete $L_\infty$ errors for series (A) for various discretisation
  parameters $h=2^{\nu_h}$ and $\varepsilon=2^{\nu_\varepsilon}$.\label{tab3}}
\end{table}

\begin{table}[htb]
\centering
\begin{tabular}{|c|c|c|c|c|c|c|c|c|}
\cline{2-9}
\multicolumn{1}{c|}{} & \multicolumn{8}{c|}{$\nu_\varepsilon$} \\
\hline
$\nu_h$ & -6 & -7 & -8 & -9 & -10 & -11 & -12 & -13 \\ \hline
-9  & 3.12e-2 &  &  &  &  &  &  &  \\ 
-10 & 2.71e-2 & 1.92e-2 &  &  &  &  &  &  \\ 
-11 & 2.69e-2 & 1.15e-2 & 1.94e-2 &  &  &  &  &  \\ 
-12 & 2.68e-2 & 1.11e-2 & 5.24e-3 & 2.44e-2 &  &  &  &  \\ 
-13 & 2.68e-2 & 1.11e-2 & 4.54e-3 & 3.17e-3 & 3.20e-2 &  &  &  \\ 
-14 & 2.68e-2 & 1.11e-2 & 4.50e-3 & 1.89e-3 & 2.86e-3 & 4.03e-2 &  &  \\ 
-15 & 2.68e-2 & 1.11e-2 & 4.50e-3 & 1.82e-3 & 8.48e-4 & 3.27e-3 & 4.70e-2 &  \\ 
-16 &  & 1.11e-2 & 4.50e-3 & 1.81e-3 & 7.34e-4 & 4.70e-4 & 3.96e-3 & 5.14e-2 \\ 
-17 &  &  & 4.50e-3 & 1.81e-3 & 7.27e-4 & 3.00e-4 & 3.47e-4 & 4.62e-3 \\ 
-18 &  &  &  & 1.81e-3 & 7.27e-4 & 2.90e-4 & 1.23e-4 & 3.26e-4 \\ \hline
\end{tabular}
\caption{Discrete $L_\infty$ errors for series (B) for various discretisation
  parameters $h=2^{\nu_h}$ and $\varepsilon=2^{\nu_\varepsilon}$.\label{tab4}}
\end{table}

The results can be found in  Tables \ref{tab3} and \ref{tab4} for 
series (A) and (B), respectively. The conditions 
of Theorem \ref{main_result} require $h\le\varepsilon^{1+3.5/\ell}$,
which is why the tables only contain entries for $h<\varepsilon$.

In the context of particle methods it is an often encountered assumption that 
convergence is achieved in the so-called {\em stationary setting},
i.e. if the ratio between $h$ and $\varepsilon$ is fixed, meaning that
approximately the same number of data sites lies in the support of
each kernel. This assumption is in particular made in almost all
application papers, though it is well-known by now (see for example
\cite{Raviart-85-1}) that this is not true. Our results corroborate
this since we can see divergence in this situation by looking at the
diagonal entries of the Tables \ref{tab3} and \ref{tab4}.

As for the general dependence of the errors on the discretisation parameters
$h$ and $\varepsilon$, the data seem to verify our findings in the 
following way: If we
look at a fixed column of either table, which corresponds to a fixed
$\varepsilon$, we see that the error becomes stationary, i.e.~further
refinement of the grid does not lead to convergence.
If we look at the rows of the table, which corresponds to an
$\varepsilon$-refinement while $h$ is kept fixed, we see that the
error eventually grows.

Finally, we have tried to estimate the crucial constants and exponents
in the error estimate using a least-squares approach.
\begin{enumerate}
 \item [(A)] Here, Theorem \ref{main_result} can be applied with 
   parameters $r,s=4$, $k,\ell=3$. In this situation, Theorem
   \ref{main_result} yields an error bound of the form 
\begin{equation}\label{errorrep}
C_1 \varepsilon^{k-0.5}+C_2\frac{h^\ell}{\varepsilon^{\ell+1.5}}
=C_1 \varepsilon^{2.5}+C_2\frac{h^{3}}{\varepsilon^{4.5}}
\end{equation}
for the approximate solution. The least-squares approximation
yields the better estimate
\[
68.7\varepsilon^{3.2} + 1.8\frac{h^{3.8}}{\varepsilon^{4.4}}.
\]
 \item [(B)] Our main result does not apply in this
   situation. Nonetheless, the data suggest that we still have
   convergence. The rows of Table \ref{tab4} indicate that for a fixed,
   small $h$ the  approximation for $\varepsilon$ to zero converges
   super-linear. A least-squares approach yields the estimate $k-0.5 =
   \sim 1.32$. To estimate the exponents in the second term of
   (\ref{errorrep}) is not feasible because there are too few values
   for a reasonable least-squares approximation.

\end{enumerate}
In both cases, the exponents for the first term are not too far off from the
value ($k-\frac{1}{2}$) that we can expect from Lemma
\ref{convolution} and Theorem
\ref{theorem_convolution_quadrature_stability} when ignoring the 
requirements on the regularity of the solution and the kernel. This
could be because the kernel and the solution are smooth except for a
finite number of points, so that errors that stem from the reduced
regularity at these points are comparatively small. The same might be
the reason why the parameters of the second term turn out to be better
than expected. 

In any case, this already indicates that the scheme should also be
useful when approximating solutions with lower regularity. This  would be 
particularly valuable in higher spatial dimensions where only very few
good methods are available to the present day. Some simple experiments
with shock formation seem to be promising and a detailed investigation
is planned for the future.

\bibliography{../../Utilities/rbf}
\end{document}